\documentclass[a4paper]{article}
\usepackage[latin1]{inputenc}
\usepackage{amsthm, amssymb,amscd,amsmath,amsfonts}
\usepackage{hyperref}
\hypersetup{colorlinks=true, citecolor=red, urlcolor=blue,linkcolor=blue}
\usepackage{color}

\newtheorem{theorem}{Theorem}
\newtheorem{proposition}{Proposition}
\newtheorem{lemma}{Lemma}
\newtheorem{remark}{Remark}
\numberwithin{equation}{section}

\newcommand{\<}{\left\langle}
\renewcommand{\>}{\right\rangle}
\def\eps{\varepsilon}
\def\N{\mathbb N}
\def\R{\mathbb R}
\def\Pb{\mathbb P}
\def\E{\mathbb{E}\,}

\newcommand{\Var}{\mbox{\rm Var}}
\newcommand{\diag}{\mbox{\rm diag}}
\def\indicator{\mathbf{1}}
\def\bP{{\mathbf P}}
\def\bY{{\mathbf Y}}
\def\bH{{\mathbf H}}
\def\bbH{{ \overline { \mathbf H}}}
\newcommand{\bhY}{\overline{\mathbf{Y}}}
\newcommand{\hY}{\overline{{Y}}}

\def\balpha{\boldsymbol\alpha}
\def\bbeta{\boldsymbol\beta}
\def\bgamma{\boldsymbol\gamma}
\def\bj{\boldsymbol j}

\author{Diego Armentano\thanks{CMAT, 
Universidad de la Rep\'{u}blica, Montevideo, Uruguay. 
E-mail: diego@cmat.edu.uy.}
\quad
Jean-Marc Aza\"{i}s\thanks{IMT, UMR CNRS 5219, Universit\'e de Toulouse, 
Email: jean-marc.azais@math.univ-toulouse.fr}
\quad
Federico Dalmao\thanks{
DMEL, 
Universidad de la Rep\'{u}blica, Salto, Uruguay. 
E-mail: fdalmao@unorte.edu.uy.}
\quad
Jos\'e R. Le\'{o}n\thanks{IMERL, Universidad de la Rep\'ublica, Montevideo, Uruguay and 
Escuela de Matem\'{a}tica. Facultad de Ciencias. 
Universidad Central de Venezuela, Caracas, Venezuela. 
E-mail: rlramos@fing.edu.uy}
}

\title{On the asymptotic variance of the number of real roots of random polynomial systems}
\begin{document}
\maketitle

\begin{abstract}
We obtain the asymptotic variance, as the degree goes to infinity, 
of the normalized number of real roots of a square 
Kostlan-Shub-Smale random polynomial system 
of any size. 
Our main tools are the Kac-Rice formula for the second factorial moment of the number of roots 
and a Hermite expansion of this random variable.\\

AMS subjet classification: Primary: 60F05,  30C15. Secondary: 60G60, 65H10
\end{abstract}

\section{Introduction}
The study of the roots of random polynomials 
is among the most important and popular topics 
in Mathematics and in some areas of Physics. 
For almost a century 
a considerable amount of literature about this problem 
has emerged from 
fields as probability, geometry, algebraic geometry, 
algorithm complexity, quantum physics, etc. 
In spite of its rich history it is still an extremely active field.

There are several reasons that lead to consider random polynomials 
and several ways to randomize them, see Bharucha-Reid and Sambandham 
\cite{brs}. 

The case of algebraic polynomials 
$P_d(t)=\sum^d_{j=1}a_jt^j$ 
with independent identically distributed  coefficients was the first one to be extensively studied 
and was completely understood during the 70s. 
If $a_1$ is centered, $\Pb(a_1=0)=0$ and $\E(|a_1|^{2+\delta})<\infty$ 
for some $\delta>0$, 
then, the asymptotic expectation and 
the asymptotic variance of the number of real roots of $P_d$, 
as the degree $d$ tends to infinity, are of order $\log(d)$ 
and, once normalized, the number of real roots converges in distribution  
towards a centered Gaussian random variable.
See the books by Farahmand \cite{far98} and Bharucha-Reid and Sambandham \cite{brs}
and the references therein for the whole picture.

The case of systems of polynomial equations 
seems to be considerably harder and has received in consequence much less attention. 
The results in this direction are confined to the Shub-Smale model 
and some other invariant distributions.
The ensemble of Shub-Smale random polynomials was introduced 
in the early 90s 
by Kostlan \cite{kostlan}. 
Kostlan argues that this is the most natural distribution 
for a polynomial system.
The exact expectation was obtained in the early 90's by geometric means, 
see Edelman and Kostlan \cite{ek} for the one-dimensional case 
and Shub and Smale \cite{ss} for the multi-dimensional one. 
In 2004, 2005 Aza\"is and Wschebor \cite{aw-pol} and Wschebor \cite{w} 
obtained by probabilistic methods 
the asymptotic variance as the number of equations and 
variables tends to infinity. 
Recently, Dalmao \cite{d} obtained the asymptotic variance and a CLT 
for the number of zeros as the degree $d$ goes to infinity in 
the case of one equation in one variable. 
Letendre in \cite{l2} studied the asymptotic behavior of  the volume of random real algebraic submanifolds.  
His results include the finiteness of the limit variance, when the degree tends to infinity, 
of the volume of the zero sets of Kostlan-Shub-Smale systems with strictly less equations than variables.
Some results for the expectation and variance of related models 
are included in \cite{aw-pol,ll,l}.

In the present paper we prove that, as the degree goes to infinity, the asymptotic variance of the normalized number of real roots of a 
Kostlan-Shub-Smale square random system with $m$ equations and $m$ variables 
exists in $(0,\infty)$. 
We use 
Rice Formulas \cite{aw} to show the finiteness of the limit variance and 
Hermite expansions as in Kratz and Le\'on \cite{kl-97} to show that it is strictly positive. 
Furthermore, we strongly exploit  the invariance under isometries 
of the distribution of the polynomials. 

The reader may wonder, in view of the results mentioned above, if the normalized number of roots satisfies a CLT when the degree of the system tends to infinity. 
The answer is affirmative if $m = 1$ \cite{d} but for the time being we cannot give an answer to this question for $m>1$. 
The ingredients to prove a CLT for a non linear functional of a Gaussian process are: 
a) to write a representation in the It\^o-Wiener chaos of the normalized functional; 
b) to demonstrate that each component verifies a CLT (Fourth Moment Theorem \cite{np}, \cite{pta}) 
and 
if the functional has an expansion involving infinitely many terms: 
c) to prove that the tail of the asymptotic variance tends uniformly (w.r.t. $d$) to zero. 
In the present case we lack a proof of c). 
For $m =  1$ the fact that the invariance by rotations is equivalent with the stationarity 
allows to build a proof similar to the one made for the number of crossings of a stationary Gaussian process. 

The rest of the paper is organized as follows. Section 2 
sets the problem and presents the main result. 
Section 3 deals with the proof and Section 4 presents some auxiliary results 
as well as the explicit form of the asymptotic variance. 

\section{Main Result}
Consider a square system $\bP$ of $m$ polynomial equations in $m$ variables 
with common degree $d>1$. 
More precisely, 
let $\bP=(P_1,\dots,P_m)$ with 
\begin{equation*}
  P_{\ell}(t)=\sum_{|\bj|\leq d}a^{(\ell)}_{\bj}t^{\bj},
\end{equation*}
where 
\begin{enumerate}
  \item $\bj=(j_{1},\dots,j_{m})\in\N^{m}$ 
and $|\bj|=\sum^{m}_{k=1}j_{k}$; 
  \item $a^{(\ell)}_{\bj}=a^{(\ell)}_{j_1\dots j_m}\in\R$, $\ell=1,\dots,m$, $|\bj|\leq d$;
  \item $t=(t_{1},\dots,t_{m})$ 
    and $t^{\bj}=\prod^m_{k=1} t^{j_k}_{k}$.
\end{enumerate}

We say that  $\bP$ has the Kostlan-Shub-Smale 
(KSS for short) distribution if the coefficients $a^{(\ell)}_{\bj}$ 
are independent centered normally distributed random variables with variances
\begin{equation*}
  \Var\left(a^{(\ell)}_{\bj}\right)=\binom{d}{\bj}=\frac{d!}{ 
j_1!\dots j_m!(d-|\bj|)!}. 
\end{equation*}

We are interested in the number of real roots of $\bP$ 
that we denote by $N^{\bP}_d$. Shub and Smale \cite{ss} proved that $\E(N^{\bP}_d)=d^{m/2}$. 
Our main result is the following.
\begin{theorem}\label{teo}
Let $\bP$ be a KSS random polynomial system with $m$ equations, $m$ variables and degree $d$. 
Then, as $d\to\infty$ we have
$$
  \lim_{d\to\infty}\frac{\Var(N^{\bP}_d)}{d^{m/2}}=V^2_\infty,
$$
where $0<V^2_\infty<\infty$.
\end{theorem}

\subsection{Explicit expression of the variance}\label{s:ex}
Using the method of section 12.1.2 of \cite{aw} 
an explicit expression for the limit variance can be given. 

For $k=1,\dots,m$ let $\xi_k,\eta_k$ be independent standard normal random vectors on $\R^k$. 
Let us define
\begin{itemize}
\item $\bar{\sigma}^2(t)=1-\frac{t^2\exp(-t^2)}{1-\exp(-t^2)}$;
\item $\bar{\rho}(t)=\frac{(1-t^2-\exp(-t^2))\exp(-t^2/2)}{1-(1+t^2)\exp(-t^2)}$;
\item $m_{k,j} = \E\left(\|\xi_k\|^j\right)=2^{j/2}\frac{\Gamma((j+k)/2)}{
\Gamma(k/2)},$
where 
$\|\cdot\|$ is the Euclidean norm on $\R^k$;
\item   for $k=1,\ldots,m-1$,  $M_k(t)=\E\left[  
\|\xi_k\|\,\|\eta_k+\frac{e^{-t^2/2}}{(1-e^{-t^2})^{1/2}}\xi_k\|\right]$; 
\item for $k=m$, 
$ M_m(t)=  
\E\left[\|\xi_m\|\, \|\eta_m+\frac{\bar{\rho}(t)}{(1-\bar{\rho}^2(t))^{1/2}}
\xi_m\|\right].$
\end{itemize}

\begin{theorem}\label{teo:var}
We have
  \begin{equation*}
    V^2_\infty
    =\frac12+\frac{\kappa_m \kappa_{m-1}}{2(2\pi)^m}
    \cdot\int^\infty_0
    t^{m-1}\left[\frac{\bar{\sigma}^4(t)(1-\bar{\rho}^2(t))}{1-e^{-t^2}}\right]^{1/2}
    \left[\prod^m_{k=1}M_k(t)-\prod^m_{k=1}m^2_{k,1}\right]dt.
  \end{equation*}
\end{theorem}\qed


\section{Proof}
\subsection{Preliminaries}
It is customary and convenient to homogenize the polynomials. 
That is, to add an auxiliary variable $t_0$ 
and to multiply the monomial in $P_\ell$ 
corresponding to the index $\bj$ by $t^{d-|\bj|}_0$. 
Let $\bY=(Y_{1},\dots,Y_{m})$ denote the resulting vector of $m$ 
homogeneous polynomials in $m+1$ real variables 
with common degree $d>1$. 
We have,
\begin{equation*}
  Y_{\ell}(t)=\sum_{|\bj|=d}a^{(\ell)}_{\bj}t^{\bj},\quad \ell=1,\dots, m,
\end{equation*}
where this time  
$\bj=(j_{0},\dots,j_{m})\in\N^{m+1}$; 
$|\bj|=\sum^{m}_{k=0}j_{k}$; 
$a^{(\ell)}_{\bj}=a^{(\ell)}_{j_0\dots j_m}\in\R$; 
$t=(t_{0},\dots,t_{m})\in\R^{m+1}$ 
and 
$t^{\bj}=\prod^m_{k=0} t^{j_k}_{k}$.  \medskip

Since $\bY$ is homogeneous, its roots consist of lines through $0$ in 
$\R^{m+1}$. Then, it is easy to check that each root of $\bP$ corresponds 
exactly to two (opposite) roots of $\bY$ on the unit sphere $S^m$ of 
$\R^{m+1}$. 
{Furthermore, one can prove that the subset of homogeneous polynomials $\bY$ 
with roots lying in the hyperplane $t_0=0$ has Lebesgue measure zero. 
Then, denoting by 
$N^{\bY}_d$ the number of roots of $\bY$ on $S^m$, 
we have $N^{\bP}_d=N^{\bY}_d/2$ almost surely.} 

From now on we work with the homogenized version $\bY$. 
The standard 
multinomial formula shows that for all $s,t\in\R^{m+1}$ we have
\begin{equation*}
r_d(s,t)
:=\E({Y_{\ell}(s)Y_{\ell}(t)})
=\<s,t\>^{d},
\end{equation*} 
where $\<\cdot,\cdot\>$ is the usual inner product in $\R^{m+1}$. 
As a consequence, we see that 
the distribution of the system $\bY$ is invariant under the action of the 
orthogonal group in $\R^{m+1}$. 
For the ease of notation we omit the dependence on $d$ of $\bY$.

\medskip

In the sequel we need to consider the derivative of $Y_\ell$, $\ell=1,\dots,m$. 
Since the parameter space is the sphere $S^m$, 
the derivative is taken in the sense of the sphere, 
that is, the spherical derivative $Y'_\ell(t)$ of $Y_\ell(t)$ is the 
orthogonal projection of the free gradient on 
the tangent space $t^\bot$ of $S^m$ at $t$. 
The $k$-th component of $Y'_\ell(t)$ at a given basis of the tangent space 
is denoted by $Y'_{\ell k}(t)$. 

The covariances between the derivatives and between 
the derivatives and the process are obtained via routine computations 
from the covariance of $Y_\ell$.
In particular, the invariance under isometries is preserved after derivation 
and for each $t\in S^m$, $\bY(t)$ is independent from
${\bY}'(t)=(Y'_1(t),\ldots,Y'_m(t))$.

\subsection{Finiteness of the limit variance}\label{s:var}
In this section we prove that 
$$
\lim_{d\to\infty}\frac{\Var(N^{\bP}_d)}{d^{m/2}}<\infty.
$$
Recall that $\E(N^{\bP}_d)=d^{m/2}$, we write 
\begin{equation}\label{Rice:second}
  \Var\big(N^{\bP}_d\big)
  =\Var\left(\frac{N^{\bY}_d}{2}\right)
  =\frac14\big[\E\big(N^{\bY}_d\big(N^{\bY}_d-1\big)\big)-\E^2\big(N^{\bY}_d\big)\big]+\frac{d^{m/2}}{2}.
\end{equation}
The quantity $\E(N^{\bY}_d(N^{\bY}_d-1))$ is computed using Rice formula \cite[Th. 6.3]{aw} and a localisation argument. 
\begin{multline*}
\E(N^\bY_d(N^\bY_d-1))
=
\int_{(S^m)^2}\E[|\det \bY'(s)\det \bY'(t)|\,|\bY(s)=\bY(t)=0]\\
\cdot p_{\bY(s),\bY(t)}(0,0)dsdt.
\end{multline*}
Here $ds$ and $dt$ are the $m$-geometric measure  on  $ S^m$ 
but  we will use in other parts $ds$ and $dt$ for the Lebesgue measure. 

The following Lemma allows us to reduce this integral 
to a one-dimensional one. 
The proof is a direct consequence of the co-area formula.
\begin{lemma}\label{lem:intACV}
Let ${\mathcal H}$ be a measurable function defined on $\R$. 
Then, we have 
\begin{multline*} 
\int_{(S^m)^2} {\mathcal H}( \langle s,t \rangle )\,ds\,dt 
= \kappa_m\kappa_{m-1}\int_0^{\pi}\sin(\psi)^{m-1} 
 {\mathcal H}(\cos(\psi))\, d\psi \\
=\frac{\kappa_m\kappa_{m-1}}{\sqrt{d}}\int_0^{ \sqrt d \pi} 
\sin\left(\frac{z}{\sqrt d}\right)^{ m-1 } 
\mathcal {\mathcal H}\left(\cos\left(\frac{z}{\sqrt d}\right)\right)\, dz,
\end{multline*}
where $\kappa_m$ is the $m$-geometric measure of $S^m$.
\qed
\end{lemma}


Let $\{e_0, e_1,\dots, e_m\}$ be the canonical basis of $\R^{m+1}$. 
Because of the invariance of $\bY$ by isometries we can assume without loss of generality that  
\begin{equation}\label{eq:st}
s = e_0, \quad t = \cos(\psi) e_0 +\sin(\psi) e_1. 
\end{equation}
For $s^\bot$ we choose as basis $\{e_1,\ldots,e_m\}$ and 
$\{ \sin(\psi) e_0 - \cos(\psi) e_1, e_2,\dots,e_m\}$ for  $t^\bot$. 
Finally, take $\psi= z/\sqrt{d}$ and use Lemma \ref{lem:intACV}. Hence,
\begin{multline*}
d^{-m/2}\E(N^{\bY}_d(N^{\bY}_d-1))\\
=\frac{\kappa_m\kappa_{m-1}}{(2\pi)^m\sqrt d}
\int_0^{\sqrt d\pi}\sin^{m-1}\left(\frac z{\sqrt d}\right)
\frac{d^{m/2}}{\left(1-\cos^{2d}(\frac z{\sqrt d})\right)^{m/2}}{{\mathcal E}\left(\frac z{\sqrt d}\right)}dz,
\end{multline*}
where ${\mathcal E}(z/\sqrt{d})$
is the conditional expectation written for $s,t$ as in \eqref{eq:st}. 

%
Now, we deal with the conditional expectation ${\mathcal E}(z/\sqrt{d})$.  
Introduce the following notation 
\begin{eqnarray}\label{eq:cs}
\mathcal{A}\left(\frac z{\sqrt d}\right)&=&-\sqrt d \cos^{d-1}\left(\frac z{\sqrt d}\right)\sin\left(\frac z{\sqrt d}\right); \\
\mathcal{B}\left(\frac z{\sqrt d}\right)&=&\cos^d\left(\frac z{\sqrt d}\right)
-(d-1)\cos^{d-2}\left(\frac z{\sqrt d}\right)\sin^2\left(\frac z{\sqrt d}\right);\notag\\
\mathcal{C}\left(\frac z{\sqrt d}\right)&=&\cos^d\left(\frac z{\sqrt d}\right);\notag\\
\mathcal{D}\left(\frac z{\sqrt d}\right)&=&\cos^{d-1}\left(\frac z{\sqrt d}\right);\notag
\end{eqnarray}
and -omitting the $(z/\sqrt{d})$-
\begin{equation*}
\sigma^2=1-\frac{\mathcal{A}^2}{1-\mathcal{C}^2},\quad
\rho  = \frac{\mathcal{B}(1-\mathcal{C}^2) -\mathcal{A}^2\mathcal{C}}{1-\mathcal{C}^2-\mathcal{A}^2}.
\end{equation*}
Thus, the variance-covariance matrix of the vector 
$\left(Y_\ell(s),Y_\ell(t),\frac{Y'_\ell(s)}{\sqrt d},\frac{Y'_\ell(t)}{\sqrt d}\right)$ 
at the given basis, can be written in the following form
\begin{eqnarray}\label{eq:c1}
\left[\begin{array}{c|c|c}
A_{11}&A_{12} &A_{13}\\ \hline
A_{12}^\top&I_m\,\,&\,A_{23}\\ \hline
A_{13}^\top&A_{23}^\top\,\,&I_m\\
\end{array}\right],
\end{eqnarray}
where $I_m$ is the $m\times m$ identity matrix, 
\begin{equation}\label{eq:c2}
A_{11}=\left[\begin{array}{cc}
1&\mathcal{C}\\
\mathcal{C}&1
\end{array}\right], \;
A_{12}= \left[\begin{array}{cccc}
0& 0&\cdots& 0\\
-\mathcal{A}&0& \cdots & 0
\end{array}\right],\:
A_{13}=  \left[\begin{array}{cccc}
\mathcal{A}& 0&\cdots& 0\\
0&0& \cdots & 0
\end{array}\right],
\end{equation}
and $A_{23}$ is the $m\times m$ diagonal matrix $\diag( \mathcal{B},\mathcal{D},\ldots,\mathcal{D}).$

Gaussian regression formulas (see \cite[Proposition 1.2]{aw}) imply that the conditional distribution of the vector 
$\big({\frac{Y'_\ell(s)}{\sqrt{d}},\frac{Y'_\ell(t)}{\sqrt{d}}}\big)$ 
(conditioned on $\bY(s)=\bY(t)=0$) 
is centered normal with variance-covariance matrix given by
\begin{equation}  \label{e:jm:b}
\left[\begin{array}{c|c}
B_{11}&B_{12} \\ \hline
B_{12}^\top&B_{22} \\
\end{array}\right],
\end{equation}
with $B_{11}=B_{22}= \diag( \sigma^2, 1,\dots,1)$ and 
$B_{12}= \diag( \sigma^2 \rho, \mathcal{D},\dots,\mathcal{D})$.

It is important to remark that if $A=(A_1\,A_2\ldots A_m)$ is a matrix with columns vectors $A_j$, 
it holds that  $\det( A)=Q_m(A_1,A_2,\dots,A_m)$ for a certain polynomial $Q_m$ of degree $m$  from $\R^{m^2}$ to $\R$.
Using representation of Gaussian vectors from a standard one  we can write 
\begin{multline*}
{{\mathcal E}\left(\frac z{\sqrt d}\right)}
=\int_{(\R^{m^2})^2}\phi_{m^2}(\mathbf x)  \phi_{m^2}(\mathbf y) 
\left| Q_m\left(\left(\begin{array}{c}
 \sigma x_{11}\\
x_{12}\\
\cdot\\
x_{1m}\end{array}\right),
\ldots,
\left(\begin{array}{c} \sigma x_{m1}\\
x_{m2}\\
\cdot\\
x_{mm}\end{array}\right) \right) \right|
\\
\left| Q_m\left(\left(\begin{array}{c}
\sigma(  \rho x_{11}+\sqrt{1-\rho^2}y_{11})  \\
\mathcal{D} x_{12}+\sqrt{1-\mathcal{D}^2}y_{12}\\
\cdot\\
\mathcal{D} x_{1m}+\sqrt{1-\mathcal{D}^2}y_{1m}\end{array}\right),
\ldots,  
\left(\begin{array}{c}
\sigma(  \rho x_{m1}+\sqrt{1-\rho^2}y_{m1}) \\
\mathcal{D} x_{m2}+\sqrt{1-\mathcal{D}^2}y_{m2}\\
\cdot\\
\mathcal{D} x_{mm}+\sqrt{1-\mathcal{D}^2}y_{mm}\end{array} \right)   
\right) \right|    d\mathbf x d\mathbf y,
\end{multline*}
where $\phi_{m^2}$ is the standard normal density in $\R^{m^2}$. 
Because of the homogeneity of the determinant we have
\begin{equation*}
{{\mathcal E}\left(\frac z{\sqrt d}\right)}
=  \sigma^2\int_{(\R^{m^2})^2}Q_m ( \mathbf{x})   Q_m ( \mathbf{z})  
\phi_{m^2}(\mathbf{x}) \phi_{m^2}(\mathbf{y})  d\mathbf x d\mathbf y
 =: \sigma^2 G(\rho,\mathcal{D}),
\end{equation*}
where $\mathbf{z}=\diag(\rho,\mathcal{D},\dots,\mathcal{D}){\mathbf{x}}+\diag(\sqrt{1-\rho^2},\sqrt{1-\mathcal{D}^2},\dots,\sqrt{1-\mathcal{D}^2}){\mathbf{y}}$.

Now, we return to the expression of the variance  in \eqref{Rice:second}. 
We have
\begin{multline}\label{m:1}
 d^{-m/2}  \Var\left(N^{\bP}_d\right)
 =\frac{1}{4 d^{m/2}}  \big[\E(N^{\bY}_d(N^{\bY}_d-1))- (\E(N^{\bY}_d))^2\big] +   \frac{1}{2}\\
=
\frac12+\frac{\kappa_m\kappa_{m-1}}{4(2\pi)^m}
\int_0^{\sqrt d\pi}\sin^{m-1}\left(\frac z{\sqrt d}\right)d^{(m-1)/2}
\\\bigg[ \frac{ \sigma^2(\frac z{\sqrt d}) } {(1-\cos^{2d}(\frac z{\sqrt d}))^{m/2}} 
G\Big(\rho\Big(\frac z{\sqrt d}\Big), \mathcal{D}\Big(\frac z{\sqrt d}\Big)\Big)- G(0,0) \bigg]dz.
\end{multline}
The proof of the convergence  of this integral is done in several steps. \medskip

In the rest of this section $\mathbf{C}$ denotes an unimportant constant, its value can change from one occurrence to another. It can depend on $m$, but recall that $m$ is fixed. \medskip

\noindent {\it Step 1:} Bounds for $G$. 

\begin{itemize}
 \item  $G(\rho,\mathcal{D})  = \int_{(\R^{m^2})^2}Q_m ( \mathbf{x})   Q_m ( \mathbf{z})  
\phi_{m^2}(\mathbf{x}) \phi_{m^2}(\mathbf{y})  d\mathbf x d\mathbf y$;
\item $G(0,0)  = \int_{(\R^{m^2})^2}Q_m ( \mathbf{x})   Q_m ( \mathbf{y})  
\phi_{m^2}(\mathbf{x}) \phi_{m^2}(\mathbf{y})  d\mathbf x d\mathbf y$;
\item $  | \sqrt{1-\rho^2} - 1| \leq  \mathbf{C} |\rho|$;  $  | \sqrt{1-\mathcal{D}^2} - 1| \leq  \mathbf{C} |\mathcal{D}|$;
\item $| Q_m( \mathbf{x})| \leq \mathbf{C} ( 1+ \|\mathbf{x}\|_{\infty})^m$;
\item any partial derivative  of $Q_m (\mathbf{w}) $ is a polynomial  of degree $m-1$  and thus it is bounded by
$\mathbf{C} ( 1+ \|\mathbf{w}\|_{\infty})^{m-1}$.
\end{itemize}
Applying that to a point between $ \mathbf{y}$ and $\mathbf{z}$, we get 
\begin{multline*}
|   Q_m ( \mathbf{z}) - Q_m ( \mathbf{y}) | \leq  \mathbf{C} (  1 + \|\mathbf{y}\|_{\infty} + \|\mathbf{z}\|_{\infty}
 )^{m-1} ( |\rho| +|\mathcal{D}|) \\
  \leq  \mathbf{C} (  1 + \|\mathbf{x}\|_{\infty} + \|\mathbf{y}\|_{\infty}
 )^{m-1} ( |\rho| +|\mathcal{D}|) ,
\end{multline*}
and
\begin{multline*}
|  Q_m ( \mathbf{x})  \cdot Q_m ( \mathbf{z}) -Q_m ( \mathbf{x})  \cdot Q_m ( \mathbf{y}) |
\\ \leq  \mathbf{C} (  1+ \|\mathbf{x}\|_{\infty})^m
( 1+   \|\mathbf{x}\|_{\infty}+\|\mathbf{y}\|_{\infty})^{m-1} ( |\rho| +|\mathcal{D}|) .
\end{multline*}
The finiteness of  all the moments of the supremum of Gaussian random variables finally yields
\begin{equation*}
|G(\rho,\mathcal{D})-G(0,0)|\le \mathbf C (|\rho |+|\mathcal{D}|).
\end{equation*}

\noindent {\it Step 2:}  Point-wise convergence.  It is a direct consequence of the expansions of  sine and cosine functions.
As $d$ tends to infinity: 
\begin{itemize}
    \item $ \mathcal{A}(\frac z{\sqrt d}) \to -z  \exp(-z^2/2)$;
    \item $ \mathcal{B}(\frac z{\sqrt d}) \to (1-z^2) \exp(-z^2/2)$;
    \item $\mathcal{C}(\frac z{\sqrt d})$ and $\mathcal{D}(\frac z{\sqrt d})$ tend to $ \exp(-z^2/2)$;
    \item $\sigma^2(\frac z{\sqrt d}) \to \frac{1-(1+z^2) \exp(-z^2)}{1- \exp(-z^2)} =\bar{\sigma}^2(z)$;
   \item  $\rho(\frac z{\sqrt d}) 
   \to \frac{(1-t^2-\exp(-t^2))\exp(-t^2/2)}{1-(1+t^2)\exp(-t^2)} =\bar{\rho}(z)$;
\end{itemize}
being $\bar{\sigma}^2$ and $\bar{\rho}$ as in Subsection \ref{s:ex}. 
This, in view of the continuity of the function $G$,  implies the point-wise convergence of the integrand in \eqref{m:1}.  \medskip

\noindent  {\it Step 3:} Symmetrization. 
We have
$\mathcal{A}(\pi-z/\sqrt{d})=(-1)^{d-1}\mathcal{A}(z/\sqrt{d})$, 
$\mathcal{B}(\pi-z/\sqrt{d}) =(-1)^{d}\mathcal{B}(z/\sqrt{d})$, 
$\mathcal{C}(\pi-z/\sqrt{d}) =(-1)^{d}\mathcal{C}(z/\sqrt{d})$, 
$\mathcal{D}(\pi-z/\sqrt{d}) =(-1)^{d-1}\mathcal{D}(z/\sqrt{d})$, 
$\sigma^2(\pi-z\sqrt{d})=\sigma^2(z/\sqrt{d})$ and 
$\rho(\pi-z\sqrt{d})=(-1)^{d}\rho(z/\sqrt{d})$. 
Hence, $B_{12}(\pi-z/\sqrt{d})$ in \eqref{e:jm:b} becomes 
$$
\big((-1)^d\sigma^2(z/\sqrt{d}) \rho(z/\sqrt{d}),(-1)^{d-1} \mathcal D(z/\sqrt{d}), \ldots,(-1)^{d-1} \mathcal D(z/\sqrt{d}) \big),
$$
the rest being unchanged.  This corresponds, for example to performing some change of signs 
(depending on the parity of $d$) on the coordinates  of $ Y'_\ell(t)$. 
Gathering the different $\ell$  this may imply a change of sign in $\det (\bY'(t))$ that plays no role because of the absolute value. As  a consequence 
 $$
 \mathcal E(\pi-z/\sqrt{d}) = \mathcal E(z/\sqrt{d}).$$

In conclusion, for the next step it suffices to dominate the integral 
in the r.h.s of \eqref{m:1} restricted to the interval $[0,\sqrt{d}\pi/2]$.\\

\noindent {\it Step 4:} Domination.
The following lemma   gives  bounds for the different terms.
\begin{lemma}  \label{l:liste} 
     There exists  some constant $\alpha$, $0<\alpha \leq 1/2 $  and  some integer $d_0$ such that  for $\frac z{\sqrt d} \leq \frac{\pi}{2}$ and $d>d_0 $: 
     \begin{itemize} 
     \item  
     $ \mathcal{C} \leq \mathcal{D} \leq \cos^{d-2} (\frac z{\sqrt d}) \leq   \exp(-\alpha z^2)$; 
     \item  $ |\mathcal{A}|\leq z \exp(-\alpha z^2)$;
     \item$ |\mathcal{B}| \leq  (1+z^2) \exp(-\alpha z^2)$;
      \item  for $ z\geq z_0$, $  1- \mathcal{C}^2  \geq 1- \mathcal{C}^2 - \mathcal{A}^2\geq \mathbf C > 0$;
   \item      $ 0\leq 1-\sigma^2 \leq   \mathbf C  \exp(-2\alpha z^2)$;
   \item $ |\rho| \leq \mathbf C      (1+z^2)^2 \exp(-2\alpha z^2)$.
\end{itemize}
\end{lemma} 
\begin{proof}
We give the proof of 1, the other cases  are similar or easier.  On $[0,\pi/2] $ there exists $ \alpha_1$, $0<\alpha_1< 1/2$ such that 
$$
   \cos(\psi)  \leq 1-\alpha_1  \psi^2.
$$
Thus,  
$$
\cos^{d-2} \bigg(\frac{ z }{\sqrt{d}} \bigg) 
\leq  
\bigg( 1-\frac{\alpha_1z^2}{d}\bigg)^{d-2}  \leq \exp \bigg( -\frac{ \alpha_1  z ^2(d-2)}{d}\bigg) 
\leq  \exp \bigg( -\alpha z^2 \bigg),
$$
as soon as $\alpha  < \alpha_1$ and $d$ is big enough.
\end{proof}

We have to  find a dominant and to prove the convergence of the  integral  at zero and at infinity.

At zero, since the function $G$ is bounded 
we have to give bounds for 
$$\frac{d^{\frac{m-1}{2}} \sin^{m-1} \Big(\frac{z}{\sqrt d}\Big) \sigma^2(\frac z{\sqrt d}) } 
{\big(1-\cos^{2d}(\frac z{\sqrt d})\big)^{m/2}}.$$
Clearly, $d^{\frac{m-1}{2}} \sin^{m-1}(z/\sqrt{d})\leq z^{m-1}$. 
Besides,
$$
\frac{ \sigma^2\left(\frac z{\sqrt d}\right)}{\big(1-\cos^{2d}(\frac z{\sqrt d})\big)^{\frac{m}{2}}} =
\frac{   1 - c^2_d(z )- c^{\prime2}_d(z )   } 
{  (1- c^2_d(z ))^{\frac{m}{2}+1}},
$$
where $c(z)=\mathcal{C}(z/\sqrt{d})$. 

For the denominator, using Lemma \ref{l:liste}, we have 
\begin{equation} \label{e:md} 
 1- c^2_d(z ) \geq \mathbf C (1- \exp(- 2 \alpha z^2)).
\end{equation}
We turn now to the numerator, 
let $X_d(.)$ be a formal Gaussian stationary process on the line  with covariance $c_d$. 
Hence,  
\begin{multline*}
1-c_d^2(z)   - c^{\prime2}_d(z )  = \Var\big( X_d(z)| X_d(0),X'_d(0)\big) 
\\
= 
\Var\big( X_d(z) - X_d(0) -zX'_d(0)| X_d(0),X'_d(0)\big) \\ 
\leq  \Var\big( X_d(z) - X_d(0) -zX'_d(0)\big) 
=z^4 \Var\Big(\int _0^1 (1-t )X''_d(ut)  dt\Big),
\end{multline*}
where we used the Taylor formula with the integral form of the remainder.  
The  covariance function $ \cos(z/\sqrt{d} )$ corresponds to the spectral measure 
$  \mu =  \frac 1 2  \big(\delta_{-d^{-1/2}} +  \delta_{d^{-1/2}}\big) $, see \cite{aw}. 
The spectral  measure  associated to $c_d (z)= \cos^d(z/\sqrt{d} )$ is the $d$-th convolution  of $\mu$ 
and a direct computation shows that its fourth spectral moment exists and is bounded uniformly in $d$. 
As a consequence, $\Var (X''_d(t) )$ is  bounded uniformely in $d$, yielding that 
\begin{equation}\label{e:mn}
   1-c_d^2(z)   - c^{\prime2}_d(z ) \leq \mathbf C  z^4.
\end{equation} 
Using  \eqref{e:md}  and  \eqref{e:mn}  we get the convergence at zero.
   
At infinity, define 
\begin{multline*}
     {\mathcal H}\left( \sigma^2\left(\frac z{\sqrt d}\right), 
     \mathcal{C}\left(\frac z{\sqrt d}\right), \rho\left(\frac z{\sqrt d}\right), \mathcal{D} \left(\frac z{\sqrt d}\right) \right)
     \\
 =\frac{ \sigma^2(\frac z{\sqrt d}) } {\big(1-\cos^{2d}(\frac z{\sqrt d})\big)^{m/2}} 
 G\left(\rho\bigg(\frac z{\sqrt d}\bigg), \mathcal{D}\bigg(\frac z{\sqrt d}\bigg)\right)dz.
\end{multline*} 
Multiplication of bounded  Lipchitz functions gives a Lipchitz function, thus
\begin{multline*}
    \left|{\mathcal H}\bigg( \sigma^2\bigg(\frac z{\sqrt d}\bigg), \mathcal{C}\bigg(\frac z{\sqrt d}\bigg), 
    \rho\bigg(\frac z{\sqrt d}\bigg), \mathcal{D} \bigg(\frac z{\sqrt d}\bigg) \bigg) 
    -{\mathcal H}(1,0,0,0) \right|\\ \leq \mathbf C \big(  |\sigma^2-1| + |\mathcal{C}| + |\rho| + |\mathcal{D} | \big).
\end{multline*}
The proof is achieved with  Lemma \ref{l:liste}.

\newpage 
\subsection{Positivity of the limit variance} 
\subsubsection{Hermite expansion of the number of real roots}
We introduce the Hermite polynomials $H_n(x)$ by $H_0(x)=1$, $H_1(x)=x$ and 
$H_{n+1}(x)=xH_n(x)-nH_{n-1}(x)$. 
The multi-dimensional versions are, 
for multi-indexes 
$\balpha=(\alpha_\ell)\in\N^m$ and 
$\bbeta=(\beta_{\ell,k})\in\N^{m^2}$, 
and vectors ${\mathbf y}=(y_\ell)\in\R^m$ 
and ${\mathbf y}'=(y'_{\ell,k})\in\R^{m^2}$
$$
\bH_{\balpha}(\mathbf y)=\prod^m_{\ell=1} H_{\alpha_\ell}(y_\ell),\quad  
\bbH_{\bbeta}(\mathbf y')=\prod^{m}_{\ell,k=1} 
H_{\beta_{\ell,k}}(y'_{\ell,k}).
$$ 
It is well known that the standardized Hermite polynomials 
$\{\frac1{\sqrt{n!}}H_n\}$, $\{\frac1{\sqrt{\balpha!}}\bH_{\balpha}\}$ 
and $\{\frac1{\sqrt{\bbeta!}}\overline{\bH}_{\bbeta}\}$ 
form orthonormal bases of the spaces $L^2(\R,\phi_1)$, $L^2(\R^m,\phi_m)$ and $L^2(\R^{m^2},\phi_{m^2})$ respectively. 
Here, $\phi_j$ stands for the standard Gaussian measure on $\R^j$ ($j=1,m,m^2$) 
and $\balpha!=\prod^m_{\ell=1}\alpha_\ell !$, $\bbeta!=\prod^m_{\ell,k=1}\beta_{\ell,k}!$. 
See \cite{np,pta} for a general picture of Hermite polynomials.

Before stating the Hermite expansion for the normalized number of roots of $\bY$ we need to introduce some coefficients. 
Let $f_{\bbeta}$ ($\bbeta\in\R^{m^2}$) 
be the  
coefficients in the Hermite's basis 
of the function 
$f:\R^{m^2}\to\R$ such that $f({\mathbf y'})=|\det(\mathbf y')|$. 
That is $f({\mathbf y'})=\sum_{\bbeta\in\R^{m^2}}f_{\bbeta}\overline{\bH}_{\bbeta}({\mathbf y'})$ with 
\begin{eqnarray}\label{eq:fb}
f_{\bbeta}&=&f_{(\bbeta_1,\ldots,\bbeta_m)}
=\frac1{\bbeta!}\int_{\R^{m^2}}|\det(\mathbf y')|\overline{\bH}_{\bbeta}(\mathbf y')\phi_{m^2}(\mathbf y')d\mathbf y' \notag\\
&=&\frac1{\bbeta_1!\ldots\bbeta_m!}\int_{\R^{m^2}}|\det(\mathbf y')
|\prod_{l=1}^m H_{\bbeta_l}(\mathbf y'_l)\frac{\exp{-\frac{||\mathbf y'_l||^2}{2}}}{(2\pi)^{\frac{m}2}}d\mathbf y'_l,
\end{eqnarray}
with 
$\bbeta_l=(\beta_{l 1},\ldots,\beta_{l m})$ and 
${\mathbf y}'_l=(y'_{l1},\dots,y'_{l m})$: $l=1,\dots,m$.

Parseval's Theorem entails
$||f||^2_2=\sum_{q=0}^\infty \sum_{|\bbeta|=q}f_{\bbeta}^2\bbeta!<\infty$. 
Moreover, 
since the function $f$ is even w.r.t. each column, the above coefficients are zero whenever 
$|\bbeta_l|$ is odd for at least one $l=1,\ldots,m.$

To introduce the next coefficients let us consider first the coefficients in the Hermite's basis in 
$L^2(\R,\phi_1)$
for the Dirac delta $\delta_0(x)$.  
They are $b_{2j}=\frac1{\sqrt{2\pi}}(-\frac12)^j\frac1{j!},$ 
and  zero for odd indices \cite{kl-97}. Introducing now the distribution $\prod_{j=1}^m\delta_0(y_j)$ and denoting as $b_{\balpha}$ its coefficients 
it holds 
\begin{equation}\label{eq:b}
b_{\balpha}=\frac1{[\frac{\balpha }2]!}\prod_{j=1}^m\frac1{\sqrt{2\pi}}\bigg[-\frac12\bigg]^{[\frac{\alpha_j}2]}
\end{equation}
or $b_{\balpha}=0$ if at least one index $\alpha_j$ is odd. 

Since the formulas for the covariances of Hermite polynomials 
work in a neater way when the underlying  random variables 
are standardized, we define the standardized derivative as
\begin{equation*}
 \hY_\ell'(t):=\frac{Y_\ell'(t)}{\sqrt{d}},\quad\mbox{and}\quad \bhY'(t):=(\hY_1'(t),\ldots,\hY'_m(t)),
\end{equation*}
where $Y_\ell'(t)$ denotes the spherical derivative of $Y_\ell$ at $t\in S^m$. 
{As said above, the $k$-th component of $\hY_\ell'(t)$ in a given basis 
is denoted by $\hY_{\ell k}'(t)$.}

\begin{proposition}\label{prop:expansion}
With the same notations as above,  we have, in the $L^2$ sense, that 
\begin{equation*}
\bar{N}_d:=\frac{N^{\bY}_d-2d^{m/2}}{2d^{m/4}}=\sum^\infty_{q=1}I_{q,d},
\end{equation*}
where
\begin{equation*}
I_{q,d}=\frac{d^{m/4}}{2}\int_{S^{m}}
\sum_{|\bgamma|=q}c_{\bgamma}\bH_{\balpha}(\bY(t))\bbH_{
\bbeta}(\bhY'(t))
dt,
\end{equation*}
with $\bgamma=(\balpha,\bbeta)\in\N^m\times\N^{m^2}$ 
and  $|\bgamma| = |\alpha|+|\beta| $ 
and $c_{\bgamma}=b_{\balpha}f_{\bbeta}$.
\end{proposition}
\begin{remark}
Hermite polynomials' properties imply that for $q\neq q'$ 
$$\E(I_{q,d}I_{q',d})=0.$$
\end{remark}
\begin{remark}
The main difficulty in order to obtain a CLT 
relies on the bound of the variance of the tail $\sum_{q\geq Q}I_{q,d}$ 
because of the degeneracy of the covariances of $(\bY,\bhY)$ near the diagonal $\{(s,t)\in S^m\times S^m:s=t\}$. 
Besides, on the sphere finding a convenient re-scaling as in the one-dimensional case \cite{d} 
is a difficult issue.   
\end{remark}

\noindent Proposition \ref{prop:expansion} is a direct consequence of the following lemma. 
\begin{lemma}
For $\varepsilon>0$ define
\begin{equation*}
N_\eps:=\int_{S^{m}}
|\det(\bY'(t))|\,\delta_\eps(\bY(t))dt,
\end{equation*}
where 
$\delta_\eps(\mathbf y) 
:=\prod^m_{\ell=1}\frac{1}{2\eps}\indicator_{\{| y_\ell|<\eps\}}$ 
for $\mathbf{y}=(y_1,\ldots,y_m)$, 
and $\bY'$ is the spherical derivative of $\bY$. 
Then, we have the following.
\begin{enumerate}
  \item For $\mathbf{v}\in\R^m$, let $N_d^{\bY}(\mathbf{v})$ denote the number of  real roots in $S^m$ of the 
    equation $\bY(t)=\mathbf{v}$. Then, $N_d^{\bY}(\mathbf{v})$ is bounded above by $2d^m$ almost surely. 
  \item $N_\eps\to N_d^{\bY}$ almost surely 
and in the $L^2$ sense as $\eps\to0$.
\item The random variable $N_{d}^{\bY}$ admits a Hermite's 
expansion.
\end{enumerate}
\end{lemma}
\begin{proof}
Since the paths of $\bY$ are smooth, Proposition 6.5 of \cite{aw} implies that 
for every $\mathbf{v} \in \R^m$ almost surely there is no point $t \in  S^m$ such that 
$\bY(t) = \mathbf{v}$ and the spherical gradient is singular. Using the local inversion theorem,  this implies that  
the roots  of $\bY = \mathbf{v}$  are isolated and by compactness they are finitely many. 
As a consequence, $N^{\bY}_d(\mathbf{v})$ is well defined and a.s. finite.
Moreover, for every $t\in\R^{m+1}$ such that $Y(t)=\mathbf{v},\, \|t\|=1$, 
we have that the set $\{Y'_1(t),\ldots, Y'_m(t),t\}$ is almost surely linearly independent in $\R^{m+1}$.
This implies that  $N^{\bY}_d(\mathbf{v})$ is uniformly bounded by the B\'ezout's number $2d^m$ concluding 1 
(see for example Milnor \cite[Lemma 1, pag. 275]{Milnor}).

By the inverse function theorem, a.s. for every regular value $\mathbf{v}\in\R^m$, $N^{\bY}_d(\cdot)$ is locally constant in a neighborhood of $\mathbf{v}$.
Furthermore, by the Area Formula (see Federer \cite{federer}, or \cite{aw} Proposition 6.1), for small $\varepsilon>0$ we have
\begin{equation}\label{eq:Nepsdef}
  N_\eps = \frac {1} {(2\eps)^m}
\int_{[-\eps,\eps]^m }N^{\bY}_d(\mathbf{v})\, d\mathbf{v},\quad a.s.
\end{equation}
Hence, 
\begin{equation} \label{e:kac}
  N^{\bY}_{d}(0)=
\lim_{\eps\to0} N_{\eps},\quad a.s.
\end{equation}
From 1. and (\ref{eq:Nepsdef}) we have $N_\eps\leq 2d^m$ a.s. Then, the convergence in \eqref{e:kac}  also happens   in $L^2$.

This convergence allows us getting a Hermite's expansion. 
We have
$$
\delta_{\eps}(\mathbf y)
=\sum_{\balpha\in\N^m} b^\eps_{\balpha}\bH_{\balpha}(\mathbf y),
$$
$$
\left|\det \left(\frac{\mathbf y'}{\sqrt 
d}\right)\right|=\sum_{\bbeta\in\N^{m^2}}f_{\bbeta}\bbH_{\bbeta}\left(\frac{\mathbf y'}{\sqrt d}\right),$$
where $b^\eps_{\balpha} $ are the Hermite  coefficients  of $\delta_\eps(\mathbf y)$     and the   $f_{\bbeta}$  have been already defined. 
Furthermore, we know  that 
$\lim_{\eps\to0}b^\eps_{\balpha}=b_{\balpha}$. 
Now, taking limit 
and regrouping terms we get as in 
Estrade and Le\'on \cite{el} that
$$
N_d=d^{m/2}\sum_{q=0}^\infty\sum_{|\balpha|+|\bbeta|=q}
b_{\balpha}f_{\bbeta}\int_{  S^{m}}
\bH_{\balpha}(\bY(t))\bbH_{\bbeta}(\bhY'(t))dt.
$$
This concludes the proof.
\end{proof}

\subsubsection{$V_\infty>0$}
To prove that $V_\infty>0$ we use 
the Hermite expansion. 
In fact,
$$
V^2_\infty=\lim_{d\to\infty}\sum^{\infty}_{q=2}\Var(I_{q,d})
\geq\lim_{d\to\infty}\Var(I_{2,d}).
$$
By Proposition \ref{prop:expansion}, we have,
\begin{equation*}
I_{2,d}
=\frac{d^{m/4}}2\sum_{|\bgamma|=2}c_{\bgamma}\int_{S^m}H_{\balpha}(\mathbf Y(t))H_{\bbeta}(\overline{\mathbf Y}'(t))dt.
\end{equation*}
The coefficients $c_{\bgamma}=b_{\balpha}f_{\bbeta}$ vanish for any odd 
$\alpha_\ell$ and $|\bbeta_{\ell}|$. 
Thus, the only possibilities to satisfy the condition $|\bgamma|=2$ 
are that either only one of the indices is $2$ and the rest vanish, or 
that $\beta_{\ell,k}=\beta_{\ell,k'}=1$ for some $k\neq k'$ and the rest vanish.
Hence, 
\begin{align*}
  I_{2,d}& =\frac{d^{m/4}}2
\int_{S^m}\Bigg[\sum^m_{\ell=1}\left( b_{2}b^{m-1}_{0}f_{(0,\ldots,0)}
H_{2}(Y_\ell(t))
+b^{m}_{0}\tilde f_{\ell 12}
H_{2}(\overline{Y}'_{\ell,1}(t))\right)\\
&+    
\sum^m_{k=2}b^{m}_{0}\tilde f_{\ell k2}
H_{2}(\overline{Y}'_{\ell,k}(t))
+
  \sum_{k\neq k'}b^{m}_{0}\tilde f_{\ell kk'1}
H_{1}(\overline{Y}'_{\ell,k}(t)) H_{1}(\overline{Y}'_{\ell,k'}(t)) 
  \Bigg]dt,
\end{align*}
where $\tilde f_{\ell k2}=f_{(0,\ldots,\beta_{\ell k},0,\ldots,0)}$, $\beta_{\ell k}=2$ 
and $\tilde f_{\ell kk'1}=f_{(0,\ldots,\beta_{\ell k},\ldots,\beta_{\ell k'}0,\ldots,0)}$, $\beta_{\ell k}=\beta_{\ell k'}=1$. 
By \eqref{eq:c1}-\eqref{eq:c2} the variables in different sums are orthogonal 
when evaluated at $s,t\in S^m$. 
Now, 
by Mehler's formula, $\E(H_2(\xi)H_2(\eta))=2(\E(\xi\eta))^2\geq 0$ 
for jointly normal variables $\xi,\eta$. 
Hence, 
bounding the sum of the variances 
by one convenient term, we have
\begin{align*}
\Var(I_{2,d})&\geq 
\Var\left(\frac{d^{m/4}}{2} b^m_0\tilde{f}_{\ell 2 2}
\int_{S^m}
H_2(\overline{Y}'_{\ell 2}(t))dt\right)\\
&=\frac{d^{m/2}}{2} (b^m_0\tilde{f}_{\ell 2 2})^2
\int_{(S^m)^2}
(\E \overline{Y}'_{\ell, 2}(s)\overline{Y}'_{\ell,2}(t))^2dsdt\\
&={(b^m_0\tilde{f}_{\ell 2 2})^2\frac{d^{m/2}}{2}}
\int_{(S^m)^2}
\bigg(\<s,t\>^{d}-(d-1)\<s,t\>^{d-2}\sqrt{1-\<s,t\>^2}\bigg)^2dsdt,
\end{align*}
where last equality is a consequence of \eqref{eq:cs}.

The integral tends to a positive limit 
as can be seen 
using Lemma \ref{lem:intACV} and the scaling $t=z/\sqrt{d}$  
as in Section \ref{s:var}. 

Finally, by \eqref{eq:b} $b_0\neq 0$. 
Besides, 
by the symmetry 
of the function $f(\cdot)=|\det(\cdot)|$ and \eqref{eq:fb}, 
$\tilde{f}_{\ell k2}=\tilde{f}_{\ell k'2}$ for all $\ell,k,k'$. 
Therefore, 
adding up \eqref{eq:fb} w.r.t. $\ell$ and $k$, we get
$$
\tilde{f}_{\ell 22}
=\frac{1}{m^2}\left(\E(|\det({\mathbf{y}'})|\|{\mathbf{y}'}\|^2_{F})-m^2\E(|\det({\mathbf{y}'})|)\right),
$$
being $\|\cdot\|_{F}$ is Frobenious' norm 
and ${\mathbf{y}'}$ an $m\times m$ standard Gaussian matrix. 
Straightforward computations using polar coordinates show that 
$\tilde{f}_{\ell 22}>0$ for all $m\geq 1$. 
This concludes the proof.\\

\bibliographystyle{amsplain}

\end{document}